\documentclass[a4paper,11pt]{amsart}

\usepackage{mathpazo}
\usepackage{amssymb}
\usepackage[pagebackref,
    ,pdfborder={0 0 0}
    ,urlcolor=black,a4paper,hypertexnames=false]{hyperref}
\hypersetup{pdfauthor=Francesca Diana and Clara L\"oh,pdftitle=The sup-semi-norm on uniformly finite homology}

\usepackage{pgf}
\usepackage{tikz}
\usetikzlibrary{decorations.pathmorphing}
        \usetikzlibrary{calc}
\usepackage[all]{xy}
\SelectTips{cm}{}

\newtheorem{thm}{Theorem}[section]
\newtheorem{prop}[thm]{Proposition}

\newtheorem{cor}[thm]{Corollary}

\theoremstyle{remark}
\newtheorem{rem}[thm]{Remark}
\newtheorem{exa}[thm]{Example}

\theoremstyle{definition}
\newtheorem{defi}[thm]{Definition}

\usepackage{stmaryrd}
\usepackage{amsfonts}
\newcommand{\Z}{\mathbb{Z}}

\newcommand{\R}{\mathbb{R}}
\newcommand{\N}{\mathbb{N}}

\newcommand{\ol}[1]{#1}

\newcommand{\UDBG}{\sf{UDBG}}

\DeclareMathOperator{\QIE}{QIE}

\DeclareMathOperator{\id}{id}

\DeclareMathOperator{\coker}{coker}

\def\epsilon{\varepsilon}
\DeclareMathOperator{\uf}{uf}

\DeclareMathOperator{\Mod}{{\sf Mod}}

\def\args{\;\cdot\;}

\def\fa#1{%
 \forall_{#1}\;\;}

\makeatletter
\newcommand\norm{\bBigg@{0.6}}

\newcommand{\supn}[2][norm]{\csname #1l\endcsname\|#2%
                           \csname #1r\endcsname\|_{\infty}}

\def\fclz#1{%
  [#1]_\Z}

\def\fclr#1{%
  [#1]_\R}

\def\huf{%
  H^{\uf}}
\def\cuf{%
  C^{\uf}}
\def\hufp{%
  H^{\uf,(+)}}
\def\cufp{%
  C^{\uf,(+)}}

\usepackage{color}
\usepackage{pdfcolmk}

\def\draftinfo{}

\author{Francesca Diana}
\author{Clara L\"oh}
\title{The $\ell^\infty$-semi-norm on uniformly finite homology}
\date{\today.\ \copyright{\ F.~Diana, C.~L\"oh 2015}. 
    This work was supported by the GRK~1692 \emph{Curvature, Cycles and Cohomology} 
    (funded by the DFG, Universit\"at Regensburg).
    \draftinfo\\
    MSC~2010 classification: 55N35, 20F65, 52C25}
\begin{document}

\begin{abstract}
  Uniformly finite homology is a coarse homology theory, defined via
  chains that satisfy a uniform boundedness condition. By
  construction, uniformly finite homology carries a canonical
  $\ell^\infty$-semi-norm. We show that, for uniformly discrete spaces
  of bounded geometry, this semi-norm on uniformly finite homology in
  degree~$0$ with $\Z$-coefficients allows for a new formulation of
  Whyte's rigidity result. In contrast, we prove that this semi-norm
  is trivial on uniformly finite homology with $\R$-coefficients in
  higher degrees.
\end{abstract}

\phantom{.}
\vspace{-1.1\baselineskip}

\maketitle

\section{Introduction}

Enriching (co)homology theories with semi-norms allows for a 
rich interaction between geometry, group theory, and topology;
a prominent example of this phenomenon is given by bounded cohomology
and simplicial volume~\cite{vbc}. In the present article, we study the
$\ell^\infty$-semi-norm on uniformly finite homology.

Uniformly finite homology is a coarse homology theory, defined via
simplicial chains that satisfy a uniform boundedness condition,
originally introduced by Block and
Weinberger~\cite{blockweinberger}. Uniformly finite homology has
various applications to amenability~\cite{blockweinberger},
rigidity~\cite{whyte}, aperiodic tilings~\cite{blockweinberger},
large-scale notions of dimension~\cite{dranishnikov}, and positive
scalar curvature~\cite{blockweinberger}.

By construction, uniformly finite homology carries additional
structure, namely a canonical $\ell^\infty$-semi-norm
(Section~\ref{sec:ufsup}). It is therefore a natural question to 
study which refined information is encoded in this semi-norm. 

In this article, we show that, for uniformly discrete spaces of
bounded geometry (UDBG~spaces, see Appendix~\ref{appx:uf}), this
semi-norm is rigid in degree~$0$ (Section~\ref{sec:rigidity}): For
example, the $\R$-fundamental class of an amenable space has
semi-norm~$1$, and the $\ell^\infty$-semi-norm with $\Z$-coefficients
allows for a new formulation of Whyte's rigidity result:

\begin{thm}[rigidity]\label{thm:rigidity}
  Let $X$ and $Y$ be UDBG spaces and let $f \colon X \longrightarrow Y$ be 
  a quasi-isometry. Then the following are equivalent:
  \begin{enumerate}
    \item The map~$f \colon X \longrightarrow Y$ is uniformly close 
      to a bilipschitz equivalence.
    \item For all~$k \in \N$, the map~$\huf_k(f;\Z) \colon
      \huf_k(X;\Z) \longrightarrow \huf_k(Y;\Z)$ is an isometric
      isomorphism with respect to the $\ell^\infty$-semi-norm.
    \item The map~$\huf_0(f;\Z) \colon \huf_0(X;\Z) \longrightarrow
      \huf_0(Y;\Z)$ is an isometric isomorphism with respect to the
      $\ell^\infty$-semi-norm.
  \end{enumerate}
\end{thm}

The corresponding equivalence for $\R$-coefficients does \emph{not}
hold (Proposition~\ref{prop:deg0R}). 

In contrast, in higher degrees, the $\ell^\infty$-semi-norm on uniformly 
finite homology (with $\R$-coefficients) is trivial: 

\begin{prop}[vanishing: non-amenable case]\label{prop:vanishingnonamenable}
  Let $X$ be a non-amenable UDBG space and let $k \in \N$. Then the
  $\ell^\infty$-semi-norm on~$\huf_k(X;\R)$ is trivial.
\end{prop}

\begin{prop}[vanishing: amenable case]\label{prop:vanishingamenable}
  Let $G$ be a finitely generated amenable group and let $k \in \N_{>0}$. Then 
  the $\ell^\infty$-semi-norm on~$\huf_k(G;\R)$ is trivial.
\end{prop}

In particular, the corresponding reduced uniformly finite homology
with $\R$-coefficients is trivial for all non-amenable UDBG spaces and
trivial in higher degrees for finitely generated amenable
groups. Moreover, the vanishing in higher degrees shows that the
$\ell^\infty$-semi-norm on uniformly finite homology is not suitable
for Engel's approach to the rough Novikov conjecture~\cite{engel}.

\subsection*{Organisation of this article}
We briefly review uniformly finite homology for UDBG spaces and its
basic properties in Appendix~\ref{appx:uf}. The $\ell^\infty$-semi-norm
on uniformly finite homology is introduced in
Section~\ref{sec:ufsup}. The rigidity result
Theorem~\ref{thm:rigidity} is proved in
Section~\ref{sec:rigidity}. The vanishing in higher degrees
(Proposition~\ref{prop:vanishingnonamenable} and
\ref{prop:vanishingamenable}) is treated in
Section~\ref{sec:vanishing}.

\section{The $\ell^\infty$-semi-norm on uniformly finite homology}\label{sec:ufsup}

We now introduce the $\ell^\infty$-semi-norm on uniformly finite homology. We start with UDBG spaces in Section~\ref{subsec:snudbg} and then discuss the case of finitely generated groups in Section~\ref{subsec:sngroups}.

Let $R$ be a ring with unit endowed with a norm $|\cdot|$ (in the
sense of Definition~\ref{defi:coefficientnorm}) and let $A$ be an
$R$-module. A \emph{norm on $A$} is a function
\[\|\cdot\|\colon A\longrightarrow\R_{\geq 0}\] satisfying the following conditions:
\begin{enumerate}
\item For all $a\in A$ we have $\|a\|=0$ if and only if $a=0$.
\item For all $a,a'\in A$ we have $\|a+a'\|\leq\|a\|+\|a'\|$.
\item For all $a\in A$, $r\in R$ we have $\|r\cdot a\|= |r|\cdot\|a\|$.
\end{enumerate}
A function $\|\cdot\|\colon A\longrightarrow\R_{\geq 0}$ is a \emph{semi-norm on $A$} if it satisfies condition (2) above and the following condition:
\begin{enumerate}
\item[(3')] For all $a\in A$, $r\in R$ we have $\|r\cdot a\|\leq |r|\cdot\|a\|$.
\end{enumerate}
\subsection{Semi-norm for UDBG spaces}\label{subsec:snudbg}

We define the $\ell^{\infty}$-norm on the module of uniformly finite chains with coefficients in a normed ring with unit. We, then, have a corresponding semi-norm on uniformly finite homology.
\begin{defi}($\ell^\infty$-semi-norm)
Let $R$ be a normed ring with unit. 
Let $X$ be a $UDBG$ space and let $n\in\N$. The \emph{$\ell^{\infty}$-norm} on $\cuf_n(X;R)$ is defined by
\begin{align*}
  \supn{\cdot} \colon \cuf_n(X;R) & \longrightarrow \R_{\geq 0} \\
  \sum_{x \in X^{n+1}} c_x \cdot x & \longmapsto \sup_{x \in X^{n+1}} |c_x|.
\end{align*}
The \emph{$\ell^{\infty}$-semi-norm} on $\huf_n(X;R)$ is the corresponding semi-norm induced on the quotient. 
More explicitly, for every~$\alpha \in \huf_n(X;R)$, it is defined by
\[
\supn{\alpha}:=\inf\bigl\{\supn{c} \bigm| c\in \cuf_n(X;R),\ \partial(c)=0,\ \alpha=[c]\bigr\}. 
\]
\end{defi}

The $\ell^{\infty}$-semi-norm on uniformly finite homology with
$\Z$-coefficients is \emph{not} homogoneous, in general
(Example~\ref{exa:nonhomog}). Moreover, the $\ell^\infty$-semi-norm on
uniformly finite homology is \emph{not} a functorial semi-norm in the
sense of Gromov~\cite{gromov,loehffsnrep} (Example~\ref{exa:nonfsn}).

\begin{prop}\label{prop:bilipisom}
  Let $R$ be a normed ring with unit. If $f \colon X \longrightarrow Y$ is a quasi-isometry between UDBG
  spaces that is uniformly close to a bilipschitz equivalence, then for each $n\in\N$ the induced map $\huf_n(f;R) \colon \huf_n(X;R)
  \longrightarrow \huf_n(Y;R)$ is an isometric isomorphism with
  respect to the $\ell^\infty$-semi-norms.
\end{prop}
\begin{proof}
In view of Proposition~\ref{prop:qi-invariance} 
it suffices to show that any bilipschitz equivalence induces an isometry in uniformly finite homology. 
Let $n\in\N$. Because a bilipschitz equivalence $g\colon X\longrightarrow Y$ is a bijection, it induces a bijection
 \begin{align*}
 X^{n+1}&\longrightarrow Y^{n+1}\\
 (x_0,\dots,x_n) &\longmapsto \bigl(g(x_0),\dots,g(x_n)\bigr).
 \end{align*}
Thus, for all $c=\sum_{\ol{x}\in X^{n+1}}c_{\ol{x}}\cdot\ol{x}\in\cuf_n(X;R)$ we have
\[
\cuf_n(g;R)(c)=\sum_{\ol{y}\in Y^{n+1}}c_{g^{-1}(\ol{y})}\cdot\ol{y}.
\]
In particular, $\supn{\cuf_n(g;R)(c)}=\supn{c}$. Therefore, $\supn{\huf_n(g;R)(\alpha)} \leq \supn{\alpha}$ for 
all~$\alpha \in \huf_n(X;R)$. Applying the same argument to the inverse of~$g$ shows that $\supn{\huf_n(g;R)(\alpha)} = \supn\alpha$ holds for all~$\alpha \in \huf_n(X;R)$.
\end{proof}

The same definition of $\ell^{\infty}$-semi-norm can be considered on uniformly finite homology for general metric spaces~\cite{blockweinberger}. Notice that for metric spaces without isolated points, the $\ell^{\infty}$-semi-norm is trivial in any degree in uniformly finite homology~\cite[Proposition 4.2.3]{dianathesis}. In particular, the corresponding reduced uniformly finite homology vanishes in any degree.

\subsection{Semi-norm for groups}\label{subsec:sngroups}

Proposition~\ref{prop:bilipisom} yields a well-defined
semi-norm on uniformly finite homology of finitely generated groups:

\begin{cor}[$\ell^\infty$-semi-norm on uniformly finite homology of groups]
  Let $R$ be a normed ring with unit. If $G$ is a finitely generated group and $S, T \subset G$ are finite
  generating sets of~$G$, then the identity map~$\id_G \colon (G,d_S)
  \longrightarrow (G,d_T)$ is a bilipschitz equivalence with respect
  to the word metrics on~$G$ associated with~$S$ and~$T$. In
  particular, the induced map~$\huf_*(G,d_S;R) \longrightarrow
  \huf_*(G,d_T;R)$ is an isometric isomorphism with respect to the
$\ell^\infty$-semi-norm. Hence, the
  $\ell^\infty$-semi-norm on~$\huf_*(G;R)$ is
  independent of the chosen generating set.
\end{cor}

\begin{defi}[$\ell^{\infty}$-semi-norm on group homology with twisted coefficients]
Let~$n\in\N$. Every chain~$c \in C_n(G;\ell^{\infty}(G,R))$ can be
written uniquely as a finite sum of the type~$\sum_{t\in
  G^n}(e,t_1,\dots,t_n)\otimes\varphi_{t}$, where almost all of
the~$\varphi_t \in \ell^{\infty}(G,R)$ are zero; we then  
define 
\[
\supn{c}:=\sup_{t \in G^n}\supn{\varphi_{t}},
\]
where for all~$t\in G^{n}$ the norm $\supn{\varphi_{t}}$ is the supremum norm on~$\ell^{\infty}(G,R)$. The \emph{$\ell^{\infty}$-semi-norm} on $H_n(G;\ell^{\infty}(G,R))$ is defined for all~$\alpha\in H_n(G;\ell^{\infty}(G,R))$ by
\[
\supn{\alpha}:=\inf\bigl\{\supn{c} \bigm| c\in C_n(G;\ell^{\infty}(G,R)),\ \partial(c)=0,\ \alpha=[c]\bigr\}. 
\]
\end{defi}
\begin{prop}\label{prop:grouptwistednorm}
Let $G$ be a group and $R$ be a normed ring with unit. The chain isomorphism~$\rho_*\colon \cuf_*(G;R)\longrightarrow C_*(G;\ell^{\infty}(G,R))$ given in Proposition~\ref{prop:grouptwisted} induces an isometric isomorphism $H_*(\rho_*)\colon\huf_*(G;R)\longrightarrow H_*(G;\ell^{\infty}(G,R))$.
\end{prop}
\begin{proof}
We prove that the  
chain isomorphism~$\rho_*$ is isometric in every degree: let~$n\in\N$ and let $c=\sum_{\ol{g}\in G^{n+1}}c_{\ol{g}}\cdot\ol{g}\in \cuf_n(G;R)$. Then with the notation from Proposition~\ref{prop:grouptwisted} we obtain
\begin{align*}
\supn{\rho_n(c)} 
& = \supn[bigg]{\sum_{t\in G^n}(e,t_1,\dots,t_n)\otimes\varphi_{c,t}}\\
&=\sup_{t\in G^n}\supn{\varphi_{c,t}}=\sup_{t\in G^n}\sup_{g\in G}|\varphi_{c,t}(g)|
=\sup_{t\in G^n}\sup_{g\in G}|c_{g^{-1}\cdot (e,t_1,\dots,t_n)}|\\ 
& = \supn{c}.
\end{align*}
Because $\rho_*$ is a chain isomorphism, it follows that for all~$\alpha\in\huf_*(G;\R)$ we have $\supn{H_*(\rho_*)(\alpha)} = \supn{\alpha}$. 
\end{proof}

For finitely generated groups, Definition~\ref{defi:amenableUDBG} gives
a notion of amenability via F\o lner sequences. Alternatively, amenable groups can be
characterised through the existence of invariant means~\cite[Theorem~4.9.2]{ceccherini}. Let~$G$ be a finitely generated amenable group and let $M(G)$ denote the set of all invariant means~$\ell^{\infty}(G;\R) \longrightarrow \R$ on~$G$. Then every mean~$m \in M(G)$ induces a transfer map~\cite[Proposition 2.15]{attie}
\[ m_* = H_*(\id_G;m) \colon \huf_*(G;\R) \cong H_*\bigl(G;\ell^\infty(G;\R)\bigr)
  \longrightarrow H_*(G;\R).
\]
In particular, in degree~$0$, we obtain a map~$m_0 \colon \huf_0(G;\R)
\longrightarrow \R$ mapping~$\fclr G$ to~$1$ with $\|m_0\|\leq 1$.

Invariant means give an alternative description of classes with trivial \mbox{$\ell^{\infty}$-se}\-mi-norm in uniformly finite homology in degree~$0$~\cite[Theorem 1]{marcinkowski}\cite[Corollary 6.7.5]{blank}:
\begin{prop}[mean-invisibility and $\ell^\infty$-semi-norm in degree~$0$]
Let $G$ be a finitely generated amenable group and let $\alpha\in\huf_0(G;\R)$. Then $\supn{\alpha}=0$ if and only if $\alpha$ is mean-invisible.
Here, a class $\alpha\in\huf_0(G;\R)$ is \emph{mean-invisible}~\cite[Definition~3.5]{blankdiana} if
\[
\fa{m\in M(G)} m_0(\alpha)=0\in\R.
\]
\end{prop}

\section{Degree~$0$: Rigidity}\label{sec:rigidity}

We will prove Theorem~\ref{thm:rigidity} in
Section~\ref{subsec:Zrigidity}. The basic idea is to identify the
fundamental class as the ``largest'' class in~$\huf_0(\args;\Z)$ of
$\ell^\infty$-semi-norm~$1$ (Proposition~\ref{prop:fclbig}), and then
to apply Whyte's rigidity result Theorem~\ref{thm:whyterigidity}. In
contrast, we will show in Section~\ref{subsec:Rnorigidity} that the
analogue of Theorem~\ref{thm:rigidity} for $\R$-coefficients does
\emph{not} hold. In Section~\ref{subsec:grouphoms}, we will translate
isometry properties of group homomorphisms into the context of
$\ell^\infty$-semi-norms on uniformly finite homology.

\subsection{Integral coefficients}\label{subsec:Zrigidity}

As a preparation for the proof of Theorem~\ref{thm:rigidity}, we look
at those classes in uniformly finite homology that can be represented
via non-negative coefficients:

\begin{defi}[non-negative classes]
  Let $X$ be a UDBG-space. We then define
  \[ \cufp_0(X;\Z) := \biggl\{
     \sum_{x \in X} c_x \cdot x \in \cuf_0(X;\Z)
     \biggm|
     \fa{x \in X} c_x \geq 0
     \biggr\}
  \]
  and we write~$\hufp_0(X;\Z) \subset \huf_0(X;\Z)$ for the subset 
  of classes that admit a representing cycle in~$\cufp_0(X;\Z)$.
\end{defi}

\begin{rem}\label{rem:hufpmaps}
  Let $f \colon X \longrightarrow Y$ be a quasi-isometry between
  UDBG-spaces. By definition of non-negative chains and the induced
  map~$\cuf_0(f;\Z)$ we obtain~$\cuf_0(f;\Z) (\cufp_0(X;\Z)) \subset
  \cufp_0(Y;\Z)$; in particular, it follows that~$\huf_0(f;\Z) (\hufp_0(X;\Z)) \subset
  \hufp_0(Y;\Z)$. Looking at a quasi-inverse of~$f$ shows that
  \[ \huf_0(f;\Z) \bigl(\hufp_0(X;\Z)\bigr) = \hufp_0(Y;\Z).
  \]
\end{rem}

\begin{prop}[the fundamental class is large]\label{prop:fclbig}
  Let $X$ be a UDBG space. 
  \begin{enumerate}
    \item Let~$\alpha \in \huf_0(X;\Z)$ with $\supn \alpha \leq 1$ and
      $\alpha \neq \fclz X$. Then there exists an element~$\beta \in
      \hufp_0(X;\Z)\setminus\{0\}$ with~$\supn{\alpha + \beta}\leq 1$.
    \item For all~$\beta \in \hufp_0(X;\Z) \setminus \{0\}$ we 
      have~$\supn{\fclz X + \beta} > 1$.
  \end{enumerate}
\end{prop}
\begin{proof}
  \emph{Ad~1.}  By assumption, there is a cycle~$c = \sum_{x \in X}
  c_x \cdot x \in \cuf_0(X;\Z)$ representing~$\alpha$ with
  $|c_x| \leq 1$ for all~$x\in X$. 
  We then consider
  \[ b := \sum_{x \in X} (1-c_x) \cdot x \in \cufp_0(X;\Z). 
  \]
  So~$\beta:= [b] = \fclz X - \alpha \neq 0 \in \hufp_0(X;\Z)$ and 
  $ \supn{\alpha + \beta} = \supn{\fclz X} \leq 1. 
  $

  \emph{Ad~2.}  \emph{Assume} for a contradiction that there is
  a~$\beta \in \hufp_0(X;\Z) \setminus\{0\}$ with~$\supn{\fclz X +
    \beta} \leq 1$.  Let $b = \sum_{x\in X} b_x \cdot x \in
  \cufp_0(X;\Z)$ be a non-negative cycle representing~$\beta$, and let
  $c = \sum_{x \in X} c_x \cdot x \in \cuf_0(X;\Z)$ be a cycle with
  $[c] = \fclz X + \beta$ and $|c_x| \leq 1$ for all~$x \in
  X$. Thus,
  \[ 0 = \fclz X + \beta - \bigl(\fclz X + \beta \bigr) 
       = \biggl[\sum_{x \in X} \bigl(1 + b_x - c_x\bigr) \cdot x\biggr].  
  \]
  In view of the vanishing criterion (Theorem~\ref{thm:deg0}) there 
  hence exist constants~$C,r \in \N$ satisfying
  \[ 
     \fa{F \subset X \text{ finite}}
     C \cdot |\partial_r F|
     \geq 
     \biggl|
     \sum_{x \in F} 1 + b_x - c_x
     \biggr|
     \geq 
     \biggl|
     \sum_{x \in F} b_x
     \biggr|.
  \]
  Therefore, the vanishing criterion implies that also~$\beta=[\sum_{x \in
      X} b_x \cdot x] = 0$ in~$\huf_0(X;\Z)$, contradicting our
  assumption on~$\beta$.
\end{proof}

\begin{proof}[Proof of Theorem~\ref{thm:rigidity}]
  The implication~``$(1) \Rightarrow (2)$'' is a consequence of
  Proposition~\ref{prop:bilipisom}, and ``$(2) \Rightarrow (3)$'' is
  trivial. We now prove~``$(3) \Rightarrow (1)$'':

  Suppose that $\huf_0(f;\Z) \colon \huf_0(X;\Z) \longrightarrow
  \huf_0(Y;\Z)$ is an isometric isomorphism. In view of Whyte's
  rigidity result Theorem~\ref{thm:whyterigidity} it suffices to show
  that~$\huf_0(f;\Z)(\fclz X) = \fclz Y$.
  
  Let $\alpha := \huf_0(f;\Z)(\fclz X) \in \hufp_0(Y;\Z)$. Because
  $\huf_0(f;\Z)$ is isometric, we have~$\supn \alpha = \supn{\fclz X}
  \leq 1$. Moreover, $\huf_0(f;\Z)$ induces a bijection
  between~$\hufp_0(X;\Z)$ and $\hufp_0(Y;\Z)$
  (Remark~\ref{rem:hufpmaps}). Hence, Proposition~\ref{prop:fclbig}(2)
  and the fact that $\huf_0(f;\Z)$ is an isometric isomorphism show
  that
  \[ \supn{\alpha + \beta} > 1
  \]
  holds for all~$\beta \in \hufp(Y;\Z) \setminus \{0\}$. Therefore,
  Proposition~\ref{prop:fclbig}(1) implies that $\alpha = \fclz Y$.
\end{proof}

The inclusion $i \colon H\longrightarrow G$ of a subgroup $H< G$ of
finite index of a finitely generated group is a quasi-isometry. If $G$
is amenable and $H$ is a proper subgroup, then $i$ is not uniformly
close to any bilipschitz equivalence~\cite[Theorem 3.5]{dymarz} (see also Corollary~\ref{cor:grouphoms}).  We use this to 
give an example of a quasi-isometry that is not uniformly close to
a bilipschitz equivalence but whose induced maps in uniformly finite
homology with integral coefficients are isometric isomorphisms in positive degrees.

\begin{exa}
As observed above, the inclusion $i\colon 2\Z\longrightarrow\Z$ is not uniformly 
close to a bilipschitz equivalence. In particular, by Theorem~\ref{thm:rigidity} 
the induced map $\huf_0(i;\Z)\colon\huf_0(2\Z;\Z)\longrightarrow\huf_0(\Z;\Z)$ 
is not an isometry with respect to the $\ell^{\infty}$-semi-norm. 
On the other hand, we will now show that for all~$k\in\N_{>0}$ the map 
$\huf_k(i;\Z)\colon\huf_k(2\Z;\Z)\longrightarrow\huf_k(\Z;\Z)$ is an isometric 
isomorphism with respect to the $\ell^{\infty}$-semi-norm. 

Using a ``fine'' (simplicial) version of uniformly finite homology~\cite[Definition 2.2]{attie}, one can easily see that for $k\in\N_{>1}$ we have $\huf_k(\Z;\Z)=0$. Indeed, $\Z$ is quasi-isometric to a uniformly contractible locally finite simplicial complex of dimension $1$ (namely, $\R$ with the standard triangulation). In particular, for all $k\in\N_{>1}$ the map~$\huf_k(i;\Z)\colon\huf_k(2\Z;\Z)\longrightarrow\huf_k(\Z;\Z)$ is trivially an isometric isomorphism with respect to the $\ell^{\infty}$-semi-norm. To prove the claim in degree $k=1$, it suffices to show that every non-trivial class in~$\huf_1(\Z;\Z)$ and in~$\huf_1(2\Z;\Z)$ has $\ell^{\infty}$-semi-norm equal to $1$ and thus the map $\huf_1(i;\Z)\colon\huf_1(2\Z;\Z)\longrightarrow\huf_1(\Z;\Z)$ is necessarily an isometric isomorphism.
 
As above, using fine uniformly finite homology, one can prove that the group~$\huf_1(\Z;\Z)$ is generated by the class $\gamma=\bigl[\sum_{z\in\Z}(z,z+1)\bigr]$. In particular, $\huf_1(\Z;\Z)\cong\Z$.
Let $\alpha\in\huf_1(\Z;\Z)\setminus\{0\}$ and let $n\in\Z\setminus\{0\}$ be such that $\alpha=n\cdot\gamma$. From Figure~\ref{fig:chainsinz} it is easy to see that
\begin{align*}
\sum_{z\in\Z}(z,z+1) - \sum_{z\in n\Z}(z,z+n)=
\sum_{z\in n\Z}\partial_{1}\biggl(& \sum_{j=0}^{n-1}(z+j,z+j+1,z+n)\\
&-(z+n,z+n,z+n)\biggr).
\end{align*}
In particular,
\[
\gamma=\biggl[\sum_{z\in n\Z}(z,z+n)\biggr].
\]
\begin{figure}
\begin{center}
\begin{tikzpicture}[scale=1.3]

	  \coordinate (A) at (0,0);
	  \coordinate (B) at ($(A)+(0:1)$);
	  \coordinate (C) at ($(B)+(0:1)$);
	  \coordinate (D) at ($(C)+(0:4)$);
	  \coordinate (E) at ($(D)+(0:1)$);
          
\draw[->] (A) -- (B);
\draw[->] (B) -- (C);
\draw[->] (D) -- (E);
\draw[->] (A) .. controls ($(A)+(15:2)$) and ($(A)+(15:4)$) .. (E);
\draw[->,dashed,black!50] (B) .. controls ($(B)+(15:1.5)$) and ($(B)+(15:3)$) .. (E);
\draw[->,dashed,black!50] (C) .. controls ($(C)+(15:1)$) and ($(C)+(15:2)$) .. (E);

	     \draw ($(A)-(0:0.2)$) node[circle,fill=black,scale=0.07] {}; 
	   	   \draw ($(A)-(0:0.4)$) node[circle,fill=black,scale=0.07] {};
	   	   \draw ($(A)-(0:0.6)$) node[circle,fill=black,scale=0.07] {};

	 \draw (A) node[circle,fill=black,scale=0.2] {}; 
	 	   	   \node[below=-1, font=\fontsize{8}{5}\selectfont] at (A) {$z$};
	\draw (B) node[circle,fill=black,scale=0.2] {}; 
		 \node[below=-1, font=\fontsize{8}{5}\selectfont] at (B) {$z+1$};
\draw (C) node[circle,fill=black,scale=0.2] {}; 
	   	   \node[below=-1, font=\fontsize{8}{5}\selectfont] at (C) {$z+2$};
	   	   \draw ($(C)+(0:0.2)$) node[circle,fill=black,scale=0.07] {}; 
	   	   \draw ($(C)+(0:0.4)$) node[circle,fill=black,scale=0.07] {};
	   	   \draw ($(C)+(0:0.6)$) node[circle,fill=black,scale=0.07] {};

	   	 \draw ($(D)-(0:0.2)$) node[circle,fill=black,scale=0.07] {}; 
	   	   \draw ($(D)-(0:0.4)$) node[circle,fill=black,scale=0.07] {};
	   	   \draw ($(D)-(0:0.6)$) node[circle,fill=black,scale=0.07] {};

\draw (D) node[circle,fill=black,scale=0.2] {}; 
	 \draw (E) node[circle,fill=black,scale=0.2] {}; 
	   	   \node[below=-1, font=\fontsize{8}{5}\selectfont] at (E) {$z+n$};

   \draw ($(E)+(0:0.2)$) node[circle,fill=black,scale=0.07] {}; 
	   	   \draw ($(E)+(0:0.4)$) node[circle,fill=black,scale=0.07] {};
	   	   \draw ($(E)+(0:0.6)$) node[circle,fill=black,scale=0.07] {};

\end{tikzpicture}
\end{center}
\caption{homologous chains in~$\cuf_1(\Z;\Z)$, schematically}
\label{fig:chainsinz}
\end{figure}

Similarly, for all $k\in\{0,\dots,n-1\}$ we
consider~$c_k:=~\sum_{z\in n\Z+k}(z,z+n)$.  Because translation by~$k$
is uniformly close to~$\id_\Z$, we have~$\gamma=[c_k]$ for all~$k \in
\{0,\dots,n-1\}$. Thus, 
\[
\alpha=n\cdot\gamma=\biggl[\sum_{k=0}^{n-1}c_k\biggr].
\]
Clearly, for any $k\in\{0,\dots, n-1\}$ we have $\supn{c_k}=1$. Moreover, the cycles $c_0, \dots, c_{n-1} \in\cuf_1(\Z;\Z)$ are supported on pairwise disjoint subsets of~$\Z^{2}$, and so we have:
\[
\supn{\alpha}=\supn[bigg]{\biggl[\sum_{k=0}^{n-1}c_k\biggr]}=1.
\]
Thus every non-trivial class in $\huf_1(\Z;\Z)$ has $\ell^{\infty}$-semi-norm~$1$. Analogously, one can prove that every non-trivial class in $\huf_1(2\Z;\Z)$ has $\ell^{\infty}$-semi-norm~$1$. Thus the map $\huf_1(i;\Z)\colon\huf_1(2\Z;\Z)\longrightarrow\huf_1(\Z;\Z)$ is an isometric isomorphism.
\end{exa}

\subsection{Real coefficients}\label{subsec:Rnorigidity}

The analogue of Theorem~\ref{thm:rigidity} for $\R$-coefficients does
not hold as the following simple example shows: 

\begin{prop}\label{prop:deg0R}
  Let $A := \{n^2 \mid n \in \N\} \subset \Z$, and consider $\Z \setminus A$ with 
  the metric induced from~$\Z$. Then the inclusion
  $i \colon \Z \setminus A \longrightarrow \Z 
  $
  is a quasi-isometry inducing for all~$k \in \N$ an isometric
  isomorphism
  \[ \huf_k(i;\R) \colon \huf_k(\Z \setminus A ; \R) 
     \longrightarrow \huf_k(\Z;\R),
  \] 
  but $i$ is \emph{not} uniformly close to a bilipschitz
  equivalence.
\end{prop}
\begin{proof}
  The inclusion~$i \colon \Z \setminus A \longrightarrow \Z$ is an
  isometric embedding with quasi-dense image, whence a
  quasi-isometry. Thus, for all~$k \in\N$ the induced map 
  $\huf_k(i;\R) \colon \huf_k(\Z \setminus A ; \R)
  \longrightarrow \huf_k(\Z;\R)$ is an isomorphism.   
  
  Moreover, $\huf_k(i;\R)$ is isometric: Let $\alpha \in \huf_k(\Z
  \setminus A;\R)$.  Because $i$ is injective, we have
  $\supn{\huf_k(i;\R)(\alpha)} \leq \supn{\alpha}$. Conversely, let $n
  \in \N_{>0}$. For~$j \in \{1,\dots,n\}$ we consider the map
  \begin{align*}
    f_j \colon \Z & \longrightarrow \Z \setminus A \\
    x & \longmapsto
    \begin{cases}
      x & \text{if $x \in \Z \setminus A$}\\
      x + j & \text{if $x \in A$ and $x \geq n^2$}\\
      -n \cdot x -j & \text{if $x \in A$ and $x < n^2$}.
    \end{cases}
  \end{align*}
  Clearly, $f_j$ is a quasi-isometric embedding and $f_j \circ i =
  \id_{\Z \setminus A}$. We then set
  \[ \varphi_* := \frac 1n \cdot \sum_{j=1}^n \cuf_*(f_j;\R) 
     \colon \cuf_*(\Z;\R) \longrightarrow \cuf_*(\Z \setminus A;\R).
  \]
  Notice that $\varphi_*$ is a chain map satisfying~$\varphi_* \circ
  \cuf_*(i;\R) = \id$.  The restrictions~$f_1|_{A}, \dots, f_n|_{A}$
  are injective and have pairwise disjoint images. By counting
  pre-images and types of elements in~$(\Z \setminus A)^{d+1}$ we
  obtain that
  \[ \|\varphi_k\| \leq 1 + \frac {2^{d+1}-1}n
  \]
  holds with respect to the corresponding $\ell^\infty$-norms. In particular, 
  \[ \supn{\alpha}
     = \supn[big]{H_0(\varphi_*)\circ \huf_0(i;\R)(\alpha)}
     \leq \Bigl(1 + \frac {2^{d+1} -1}n\Bigr) \cdot \supn[big]{\huf_0(i;\R)(\alpha)}.
  \]
  Taking the infimum over all~$n \in \N_{>0}$ gives the desired estimate.

  On the other hand, Whyte's vanishing criterion
  (Theorem~\ref{thm:deg0}) shows that~$[\chi_A] \neq 0$
  in~$\huf_0(\Z;\Z)$. Hence, in~$\huf_0(\Z;\Z)$ we obtain
  \[ \huf_0(i;\Z) \fclz{\Z \setminus A} = \fclz \Z - [\chi_A] \neq \fclz \Z, 
  \]
  and so $i$ is \emph{not} uniformly close to a bilipschitz equivalence 
  (Theorem~\ref{thm:whyterigidity}).
\end{proof}

As we have seen in Proposition~\ref{prop:fclbig} in the case of
$\Z$-coefficients, for an amenable UDBG space $X$ we have $\supn{\fclz
  X}\leq 1$. More precisely, since a non-trivial class in uniformly
finite homology with $\Z$-coefficients cannot have
$\ell^{\infty}$-semi-norm strictly smaller than 1, we have
$\supn{\fclz X}= 1$. Similarly, in the amenable case also the
fundamental class with \mbox{$\R$-co}\-efficients is rigid with 
respect to the $\ell^\infty$-semi-norm in the following sense:

\begin{prop}[semi-norm of the fundamental class]\label{prop:fclsemi-norm}
Let $X$ be an amenable UDBG space and let $\fclr X\in\huf_0(X;\R)$ be its fundamental class. 
Then 
\[ \supn{\fclr X}= 1.\]
\end{prop}
\begin{proof}
By definition of the fundamental class, we have~$\supn{\fclr X}\leq
1$. For the converse estimate, we use an averaging argument: Let
$(S_n)_{n \in \N}$ be a F\o lner sequence for~$X$
(Definition~\ref{defi:amenableUDBG}), and let $\omega$ be a
non-principal ultrafilter on~$\N$. A straightforward calculation using
the F\o lner condition shows that the map
\begin{align*}
  \cuf_0(X;\R) & \longrightarrow \R \\
  \sum_{x\in X} c_x \cdot x & \longmapsto \lim_{n \in \omega} \frac1{|S_n|} \cdot \sum_{x \in S_n} c_x
\end{align*}
induces a well-defined map~$m \colon \huf_0(X;\R) \longrightarrow \R$ with
$m(\fclr X) = 1$ and
\[ \fa{\alpha \in \huf_0(X;\R)} \bigl| m(\alpha) \bigr| \leq \supn \alpha. 
\]
Hence, $\supn{\fclr X} \geq 1$. 
\end{proof}

This semi-norm rigidity of the fundamental class in the amenable case
leads to the following examples:

\begin{exa}\label{exa:nonhomog}
The $\ell^{\infty}$-semi-norm on uniformly finite homology
with $\Z$-co\-ef\-ficients is \emph{not} homogeneous, in general. For
example: In~$\huf_0(\Z;\Z)$ we have
\[ 2 \cdot [\chi_{2\Z}] 
   = [\chi_{2 \Z}] + [\chi_{2\Z + 1}] 
   = \fclz \Z,
\] 
which is non-trivial because $\Z$ is amenable (Theorem~\ref{thm:charamenability}). 
So, $[\chi_{2\Z}] \neq 0$ in~$\huf_0(\Z;\Z)$ and $\supn[big]{[\chi_{2\Z}]}=1=\supn[big]{\fclz \Z}$, but
\[
\supn[big]{2\cdot[\chi_{2\Z}]}=\supn[big]{\fclz \Z}=1< 2=2\cdot\supn[big]{[\chi_{2\Z}]}. 
\]
\end{exa}

\begin{exa}\label{exa:nonfsn}
  The $\ell^\infty$-semi-norm is \emph{not} a functorial semi-norm
  on~$\huf_*(\args;\Z)$ and~$\huf_*(\args;\R)$ in the sense of
  Gromov~\cite{gromov,loehffsnrep}: For example, the quasi-isometric embedding
  \begin{align*}
    f \colon \Z & \longrightarrow \Z \\
    n  & \longmapsto \Bigl\lfloor \frac n 2 \Bigr\rfloor
  \end{align*}
  does not satisfy~$\|\huf_0(f;\Z)\| \leq 1$ or $\|\huf_0(f;\R)\| \leq 1$ because
  \begin{align*}
    \supn[big]{\huf_0(f;\R)\fclr \Z} 
    & = \supn[big]{2 \cdot \fclr \Z} 
    = 2 \cdot \supn[big]{\fclr \Z} = 2\\ 
    & > 1 = \supn[big]{\fclr \Z},
    \\
    \supn[big]{\huf_0(f;\Z)\fclz \Z} 
    & = \supn[big]{2 \cdot \fclz \Z} 
    \geq \supn[big]{2 \cdot \fclr \Z} = 2 \\
    & > 1 = \supn[big]{\fclz \Z}.
  \end{align*}
\end{exa}

\subsection{Group homomorphisms}\label{subsec:grouphoms}

We can translate a result by Dymarz~\cite{dymarz} into the context of
the $\ell^\infty$-semi-norm on uniformly finite homology in
degree~$0$. In particular, we can characterise when a group
homomorphism between finitely generated amenable groups with finite
kernel and cokernel is uniformly close to a bilipschitz equivalence
using isometric isomorphisms on uniformly finite homology:

\begin{cor}\label{cor:grouphoms}
  Let $G$, $H$ be finitely generated amenable groups, let
  $f\colon G \longrightarrow~H$ be a homomorphism with finite kernel
  and finite cokernel. Then the following are equivalent:
  \begin{enumerate}
    \item\label{it:kercoker} We have $|\ker(f)|=|\coker(f)|$.
    \item\label{it:bilip} The map $f$ is uniformly close to a bilipschitz equivalence.
    \item\label{it:zisom} The induced map
      $\huf_0(f;\Z)\colon\huf_0(G;\Z)\longrightarrow\huf_0(H;\Z)$ is
      an isometric isomorphism with respect to the
      $\ell^{\infty}$-semi-norm.
    \item\label{it:risom} The induced map
      $\huf_0(f;\R)\colon\huf_0(G;\R)\longrightarrow\huf_0(H;\R)$ is
      an isometric isomorphism with respect to the
      $\ell^{\infty}$-semi-norm.
  \end{enumerate}
\end{cor}

\begin{proof}
  Notice that $f$ is a quasi-isometry because $f$ has finite kernel
  and cokernel.  The equivalence~``$(\ref{it:kercoker})\Leftrightarrow
  (\ref{it:bilip})$'' is a result of
  Dymarz~\cite[Theorem~3.6]{dymarz}. The
  equivalence~``$(\ref{it:bilip})\Leftrightarrow (\ref{it:zisom})$''
  follows from Theorem~\ref{thm:rigidity}.  The
  implication~``$(\ref{it:bilip})\Rightarrow (\ref{it:risom})$'' is a
  consequence of Proposition~\ref{prop:bilipisom}.  The
  implication~``$(\ref{it:risom}) \Rightarrow (\ref{it:kercoker})$''
  follows from the fact that the $\R$-fundamental class has
  $\ell^{\infty}$-semi-norm equal to~$1$
  (Proposition~\ref{prop:fclsemi-norm}) and that
  \[ \huf_0(f;\R)(\fclr G) = \frac{|\ker(f)|}{|\coker(f)|} \cdot \fclr H.
  \qedhere
  \]
\end{proof}

\section{Higher degree: Vanishing}\label{sec:vanishing}

In this section we will show that the $\ell^\infty$-semi-norm on uniformly 
finite homology with $\R$-coefficients is trivial in higher degrees. 

\subsection{The non-amenable case}

We begin with the proof of 
Proposition~\ref{prop:vanishingnonamenable}: To this end, we shrink
the involved coefficients by spreading the chain over the space;
non-amenability allows us to keep control over the
$\ell^\infty$-semi-norms of the classes in question.

\begin{proof}[Proof of Proposition~\ref{prop:vanishingnonamenable}]
  We consider the double~$Y := X \times \{0,1\}$ of~$X$ with respect 
  to the sum metric
  \begin{align*}
    Y \times Y & \longrightarrow \R_{\geq 0}\\
    \bigl( (x,j), (x',j') 
    \bigr)
    & \longmapsto d(x,x') + |j - j'|,
  \end{align*}
  where $d$ is the metric on~$X$. Then $Y$ is a UDBG space and 
  \begin{align*}
    p \colon Y & \longrightarrow X \\
    (x,j) & \longmapsto x
  \end{align*}
  is a quasi-isometry (for example, a quasi-inverse is given by the
  inclusion into the $0$-factor). Because $X$ and hence also $Y$ are
  non-amenable, $p$ is uniformly close to a bilipschitz equivalence
  (Corollary~\ref{cor:rignonamenable}). Hence, for all~$k \in \N$ the
  induced map~$\huf_k(p;\R) \colon \huf_k(Y;\R) \longrightarrow
  \huf_k(X;\R)$ is an isometric isomorphism
  (Proposition~\ref{prop:bilipisom}). Let $\alpha \in \huf_k(X;\R)$. 
  Then 
  \[ \beta := \frac 12 \cdot \alpha \times 0 
            + \frac12 \cdot \alpha \times 1 
     \in \huf_k(Y;\R)
  \]
  satisfies~$\huf_k(p;\R)(\beta) = 1/2\cdot \alpha + 1/2 \cdot \alpha
  = \alpha$ and  \begin{align*}
    \supn \alpha 
    = \supn[big]{\huf_k(p;\R)(\beta)}
    = \supn \beta
    = \supn[Big]{\frac12 \cdot \alpha \times 0 + \frac12 \cdot \alpha \times 1}
    \leq \frac12 \cdot \supn\alpha.
  \end{align*}
  Therefore, we obtain~$\supn \alpha = 0$.
\end{proof}

\subsection{The amenable case}

Finally, we will prove Proposition~\ref{prop:vanishingamenable}: The
key is to use the interpretation of uniformly finite homology of
groups in terms of group homology with $\ell^\infty$-coefficients and
to apply a vanishing result on $\ell^1$-homology.

\begin{proof}[Proof of Proposition~\ref{prop:vanishingamenable}]
  We have~$\huf_k(G;\R) \cong H_k(G;\ell^\infty(G;\R))$
  and the corresponding isomorphism is isometric 
  with respect to 
  the \mbox{$\ell^\infty$-semi-norms} 
  (Proposition~\ref{prop:grouptwistednorm}). Hence, it suffices to
  show that the $\ell^{\infty}$-semi-norm~$\supn{\args}$ is trivial on~$H_k(G;\ell^\infty(G;\R))$.

  Clearly, $\supn{\,\cdot\,} \leq \|\cdot\|_1$
  on~$H_k(G;\ell^\infty(G;\R))$, where $\|\cdot\|_1$ is the
  $\ell^1$-semi-norm induced from the bar
  complex. However, since $k>0$ and $G$ is amenable we
  know that $\|\cdot\|_1 = 0$ on~$H_k(G;\ell^\infty(G;\R))$: In fact,
  the comparison map
  \[ H_k\bigl(G;\ell^\infty(G;\R)\bigr) 
    \longrightarrow H_k^{\ell^1}\bigl(G;\ell^{\infty}(G;\R)\bigr)
  \]
  is isometric with respect to the $\ell^1$-semi-norm and
  $H_k^{\ell^1}(G;\ell^{\infty}(G;\R))$ is trivial~\cite[Proposition~2.4,
    Corollary~5.5]{loehl1}.
\end{proof}

\appendix
\section{Review of uniformly finite homology}\label{appx:uf}

We start by recalling the category of UDBG spaces
(Section~\ref{subsec:udbg}) and the definition of uniformly finite
homology for UDBG~spaces (Section~\ref{subsec:ufdef}) and for finitely
generated groups (Section~\ref{subsec:ufgroups}). In
Section~\ref{subsec:uffcl}, we review the fundamental class in
uniformly finite homology and its relation with rigidity and
amenability.

\subsection{UDBG spaces}\label{subsec:udbg}

For simplicity, we will work in the category of UDBG spaces. We briefly 
recall the definitions:

\begin{defi}[UDBG space]
  A metric space~$X$ is a \emph{UDBG space} if it is uniformly discrete 
  and of bounded geometry, i.e., if
  \begin{itemize}
\item There exists $\epsilon >0$ such that 
\[
\fa{x,y\in X} d(x,y)<\epsilon \Longleftrightarrow x=y.
\]
\item For every~$r\in\R_{>0}$ there exists $K >0$ such that
\[
\fa{x\in X} \left|B_r(x)\right|<K.
\]
\end{itemize}
\end{defi}

\begin{exa}
  Every finitely generated group equipped with some word metric of a
  finite generating set is a UDBG space. The space $\bigl\{n^2 \bigm|
  n\in\N\bigr\}$ with the metric induced by the standard metric in
  $\R$ is a UDBG space.
  Every manifold with a Riemannian metric of bounded geometry is quasi-isometric to a UDBG space contained in it (namely the space of vertices of a suitable triangulation)~\cite[Theorem 1.14]{attie}.
  \end{exa}

\begin{defi}[quasi-isometric embedding, quasi-isometry, and bilipschitz equivalence]\label{defi:qi_and_be}
Let $(X,d_X)$ and $(Y,d_Y)$ be UDBG spaces and let $f\colon X\longrightarrow Y$ be a map.
\begin{itemize}
\item We say that $f$ is a \emph{quasi-isometric embedding} if there are~$C,D\in\R_{>0}$ such that for all $x,x'\in X$ we have
\[
 \frac{1}{C}\cdot d_{X}(x,x')-D\leq d_{Y}\bigl(f(x),f(x')\bigr)\leq C\cdot d_{X}(x,x')+D.
\]
\item The map~$f$ is \emph{uniformly close} to a map~$f' \colon X \longrightarrow Y$ if 
\[
\sup_{x \in X} d_Y\bigl(f(x),f'(x)\bigr) < \infty
\]
\item We say that $f$ is a \emph{quasi-isometry} if it is a quasi-isometric embedding and if it admits a quasi-inverse quasi-isometric embedding, i.e., if there exists a quasi-isometric embedding~$g\colon Y\longrightarrow X$ such that $f\circ g$ is uniformly close to $\id_Y$ and $g\circ f$ is uniformly close to $\id_X$.
\item The map~$f$ is a \emph{bilipschitz embedding} if there is a~$C\in\R_{>0}$ such that
\[
\fa{x,x'\in X} \frac{1}{C}\cdot d_{X}(x,x')\leq d_{Y}\bigl(f(x),f(x')\bigr)\leq C\cdot d_{X}(x,x').
\]
\item We say that $f$ is a \emph{bilipschitz equivalence} if it is a bilipschitz embedding and if there is a bilipschitz embedding $g\colon Y\longrightarrow X$ such that $f\circ g=\id_Y$ and $g\circ f=\id_X$. (For UDBG~spaces 
this is equivalent to $f$ being a bijective quasi-isometry.)
\end{itemize}
\end{defi}

\begin{defi}(category of UDBG spaces)
We define the category $\UDBG$ as follows:
\begin{itemize}
\item The objects in $\UDBG$ are UDBG spaces.
\item The set of morphisms between two objects $X$ and $Y$ in $\UDBG$ is given by the set \[\QIE(X,Y):=\bigl\{f\colon X\longrightarrow Y \bigm|  f \; \text{quasi-isometric embedding}\bigr\}\big{/}\sim\]
where $f\sim f'$ if and only if $f$ is uniformly close to $f'$. 
\item Composition of morphisms is given by ordinary composition of representatives.
\end{itemize}
\end{defi}

Clearly, quasi-isometries of UDBG~spaces correspond to isomorphisms in the category~$\UDBG$.

\subsection{Uniformly finite homology of UDBG spaces}\label{subsec:ufdef}

Uniformly finite chains are combinatorial infinite chains on UDBG
spaces that satisfy certain geometric finiteness conditions. We consider uniformly finite chains with coefficients in a normed ring with unit.

\begin{defi}[normed ring]\label{defi:coefficientnorm}
Let $R$ be a ring with unit. A \emph{norm on $R$} is a function
$
|\cdot|\colon R \longrightarrow\R_{\geq 0}
$
satisfying the following conditions:
\begin{enumerate}
\item For all~$r\in R$ we have $|r|=0$ if and only if $r=0$.
\item For all~$r,r'\in R$ we have $|r+r'|\leq |r|+|r'|$.
\item For all~$r,r'\in R$ we have $|r\cdot r'|=|r|\cdot |r'|$.
\end{enumerate}
\end{defi}

\begin{defi}(uniformly finite homology)\label{defi:uf}
Let $R$ be a normed ring with unit and $X$ be a UDBG space. For each~$n\in\N$ the space of \emph{uniformly finite $n$-chains} is the $R$-module~$\cuf_n(X;R)$ whose elements are functions of type~$X^{n+1} \longrightarrow R$, written 
as formal sums of the form
\[
c=\sum_{\ol{x}\in X^{n+1}}c_{\ol{x}}\cdot\ol{x},
\]
satisfying the following conditions:
\begin{itemize}
\item For any $\ol{x}\in X^{n+1}$, we have $c_{\ol{x}}\in R$. 
Moreover, there exists a constant~$K\in\R_{>0}$ such that
\[
\fa{\ol{x}\in X^{n+1}} |c_{\ol{x}}|<K
\]
\item There exists a constant $R\in\R_{>0}$ such that:
\[
\fa{\ol{x}=(x_0,\dots,x_n)\in X^{n+1}} \sup_{i,j\in\{0,\dots,n\}} d(x_i,x_j)>R \Longrightarrow c_{\ol{x}}=0.
\]
\end{itemize}
For each $n\in\N$, we define a boundary operator 
\[
\partial_n\colon \cuf_n(X;R)\longrightarrow\cuf_{n-1}(X;R)
\]
that takes every $\ol{x}\in X^{n+1}$ to 
\[
\partial_n(\ol{x})=\sum_{j=0}^{n}(-1)^{j}\cdot (x_0,\dots,\widehat{x}_j,\dots,x_n)
\] 
and is extended in the obvious way to all of~$\cuf_n(X;R)$; this map is indeed well-defined. 
Moreover, for each $n\in\N$ we have $\partial_{n}\circ\partial_{n+1}=0$. In this way we get a well-defined chain complex. 
The homology of $(\cuf_n(X;R),\partial_n)_{n\in\N}$ is the \emph{uniformly finite homology of $X$} and it is denoted by $\huf_*(X;R)$.
\end{defi}

Block and Weinberger~\cite[Proposition 2.1]{blockweinberger} observed that uniformly finite homology 
is quasi-isometry invariant:

\begin{prop}[quasi-isometry invariance]\label{prop:qi-invariance}
Let $R$ be a normed ring with unit. 
Let $X,Y$ be UDBG spaces and let~$f\colon X\longrightarrow Y$ be a quasi-isometric embedding. Then $f$ induces a chain map, defined for each~$n\in\N$ by 
\begin{align*}
\cuf_n(f;R)\colon\cuf_n(X;R)&\longrightarrow \cuf_n(Y;R)\\
 \sum_{\ol{x}\in X^{n+1}} c_{\ol{x}}\cdot\ol{x} &\longmapsto \sum_{\ol{x}\in X^{n+1}}c_{\ol{x}}\cdot \bigl(f(x_0), \dots, f(x_n)\bigr). 
\end{align*}
If $f$ is uniformly close to a quasi-isometric embedding $f'\colon X \longrightarrow Y$, then \[\huf_*(f;R)=\huf_*(f';R)\colon \huf_n(X;R)\longrightarrow \huf_n(Y;R).\]
In particular, any quasi-isometry induces an isomorphism in uniformly finite homology.
\end{prop}

In view of Proposition~\ref{prop:qi-invariance}, uniformly finite homology 
with coefficients in a normed ring~$R$ is a functor from the category $\UDBG$ to the category~$\Mod_*^R$ 
of graded $R$-modules:
\[
\begin{array}{cccc}
\huf_*(\args;R)\colon & \UDBG & \longrightarrow & \Mod_*^R\\
\text{on objects}\colon & X & \longmapsto & \huf_*(X;R)\\
\text{on morphisms}\colon & [f]\colon X\rightarrow Y & \longmapsto & \huf_*(f;\R)\colon\huf_*(X;R)\rightarrow\huf_*(Y;R).
\end{array}
\]

Uniformly finite homology was introduced by Block and Weinberger 
to study the large scale structure of metric spaces of bounded geometry~\cite{blockweinberger,nowakyu}. 
One of the main applications provided by Block and Weinberger is a 
characterisation of amenability for metric spaces of 
bounded geometry~\cite[Theorem 3.1]{blockweinberger}. 
Whyte used uniformly finite homology to develop a criterion 
to distinguish between quasi-isometries and bilipschitz equivalences 
in the case of UDBG spaces. We recall these two applications 
in Section~\ref{subsec:uffcl} 
(Theorem~\ref{thm:whyterigidity} and Theorem~\ref{thm:charamenability}).

\subsection{Uniformly finite homology of groups}\label{subsec:ufgroups}

As a consequence of quasi-isometry invariance (Proposition~\ref{prop:qi-invariance}) we obtain 
that for finitely generated groups uniformly finite homology is independent from the chosen word metric.
\begin{cor}
Let $G$ be a finitely generated group and let $d_S,d_T$ be the word metrics on $G$ associated to finite generating sets
 $S,T\subset G$. Then the identity map~$\id_G \colon (G,d_S)
  \longrightarrow (G,d_T)$ is a bilipschitz equivalence (whence a quasi-iso\-metry). Thus, the induced map~$\huf_*(\id_G;R)\colon\huf_*(G,d_S;R) \longrightarrow
  \huf_*(G,d_T;R)$ is an isomorphism for every normed ring~$R$ with unit.
  \end{cor}

We now recall homology of groups with twisted coefficients. Proposition~\ref{prop:grouptwisted} shows that in the case of finitely generated groups uniformly finite homology is isomorphic to group homology with coefficients in the module of bounded functions on the group.

\begin{defi}[group homology] 
Let $G$ be a group and $R$ be a ring with unit. We consider $\bigl(C_n(G;R),\partial_n\bigr)_{n\in\N}$ to be the $R[G]$-chain complex, where for each $n\in\N$ we have:
\begin{itemize}
\item The module $C_n(G;R)$ is the free $R$-module with the basis~$G^{n+1}$ (with the $R[G]$-structure induced from the diagonal action on~$G^{n+1}$).
\item The operator $\partial_n$ is the standard boundary map given by 
\begin{align*}
\partial_n\colon C_n(G;R)&\longrightarrow C_{n-1}(G;R)\\
(g_0,\dots,g_n)&\longmapsto \sum_{j=0}^{n}(-1)^{j}\cdot (g_0,\dots,\widehat{g}_j,\dots,g_n).
\end{align*}
\end{itemize}
Let $A$ be a left $R[G]$-module and let $\overline{C}_*(G;R)$ be the right $R[G]$-module obtained from~$C_*(G;R)$ via the canonical involution $G\rightarrow G$, $g\mapsto g^{-1}$. The \emph{homology of $G$ with coefficients in $A$} is the homology of the $R$-chain complex~$C_*(G;A) := \overline C_*(G;R) \otimes_{R[G]} A$.
\end{defi}

Let $R$ be a ring with unit endowed with a norm $|\cdot|$. The space $\ell^{\infty}(G,R)$ of functions $\varphi\colon G\rightarrow R$ that are bounded with respect to the supremum norm~$\supn{\varphi}:= \sup_{g\in G}|\varphi(g)|$ has a natural $R[G]$-module structure with respect to the action
\begin{align*}
G\times\ell^{\infty}(G,R) &\longrightarrow \ell^{\infty}(G,R)\\
(g,\varphi) &\longmapsto \bigl(g\cdot\varphi\colon g'\longmapsto \varphi(g^{-1}\cdot g')\bigr).
\end{align*}

Notice that, in the case of uniformly finite homology, the simplices of a given uniformly finite chain are tuples in~$G^{n+1}$ of uniformly bounded diameter; therefore they are contained in the $G$-orbit of finitely many simplices of the form $(e,t_1,\dots,t_n)\in G^{n+1}$. 
Hence, we have~\cite{bnw}\cite[Proposition~2.2.4]{dianathesis}:
\begin{prop}[uniformly finite homology as group homology]\label{prop:grouptwisted}
Let $G$ be a finitely generated group endowed with the word metric with respect to some finite generating set and let $R$ be a normed ring with unit. For~$n\in\N$ consider 
\begin{align*}
\rho_n\colon \cuf_n(G;R) & \longrightarrow C_n\bigl(G;\ell^{\infty}(G,R)\bigr)\\
\sum_{\ol{g}\in G^{n+1}}c_{\ol{g}}\cdot\ol{g} &\longmapsto \sum_{t=(t_1,\dots,t_n)\in G^n}(e,t_1,\dots,t_n)\otimes\varphi_{c,t}
\end{align*}
where for all~$t\in G^{n}$ the map $\varphi_{c,t}\in\ell^{\infty}(G,R)$ is given by
\[
\varphi_{c,t}\colon g\longmapsto c_{g^{-1}\cdot(e,t_1,\dots,t_n)}.
\]
Then $\rho_* \colon \cuf_*(G;R) \longrightarrow
C_*(G;\ell^\infty(G,R))$ is a chain isomorphism; in particular, $\rho_*$ 
induces an isomorphism
$
H_*(\rho_*) \colon \huf_*(G;R)\longrightarrow H_*(G;\ell^{\infty}(G,R)).
$
\end{prop}

\subsection{The fundamental class in uniformly finite homology}\label{subsec:uffcl}

In degree~$0$ there is a canonical uniformly finite homology class, the fundamental class:

\begin{defi}[fundamental class]
Let $X$ be a UDBG space and let $R$ be a normed ring with unit. The \emph{$R$-fundamental class of $X$ in~$\huf_0(X;R)$} is the class~$[X]_R\in\huf_0(X;R)$ represented by the cycle $\sum_{x\in X}1 \cdot x\in \cuf_0(X;R)$. 
\end{defi}

We recall now a central application of uniformly finite homology, due to Whyte~\cite[Theorem 1.1]{whyte}, namely a criterion to distinguish between quasi-isometries and bilipschitz equivalences in the case of UDBG spaces.
\begin{thm}[bilipschitz equivalence rigidity]\label{thm:whyterigidity}
Let $X,Y$ be UDBG spaces and let $f\colon X\longrightarrow~Y$ be a quasi-isometry. 
Then $f$ is uniformly close to a bilipschitz equivalence if and only if $\huf_0(f;\Z)(\fclz X)=\fclz Y$.
\end{thm}

One step in Whyte's proof is the following characterisation of the trivial class in
uniformly finite homology in degree~$0$~\cite[Theorem~7.6]{whyte}:
\begin{thm}[vanishing criterion in degree~$0$]\label{thm:deg0}
  Let $X$ be a UDBG space, and let $c = \sum_{x\in X} c_x \cdot x$ be a
  cycle in~$\cuf_0(X;\Z)$.  Then $[c] = 0 \in \huf_0(X;\Z)$ if and
  only if there exist constants~$C, r \in \N$ such that for all finite
  subsets~$F \subset X$ we have
  \[ C \cdot |\partial_r F| \geq \biggl| \sum_{x \in F} c_x \biggr|.
  \]
\end{thm}

These boundary conditions are closely related to amenability;  
let us recall the definition of amenability for UDBG spaces: 
\begin{defi}[amenable UDBG space]\label{defi:amenableUDBG}
A UDBG space $X$ is \emph{amenable} if it admits a \emph{F\o lner sequence}, i.e., sequence~$(S_n)_{n\in\N}$ of non-empty finite subsets~$S_n\subset X$ such that 
\[
\fa{r\in\R_{>0}} \lim_{n\to\infty}\frac{\left|\partial_r(S_n)\right|}{|S_n|}=0.
\]
\end{defi}
\begin{rem}
We can consider finitely generated groups as UDBG spaces by endowing them with word metrics for finite generating sets. In this case, Definition~\ref{defi:amenableUDBG} is equivalent to the standard definition of amenability using F\o lner sequences~\cite[Definition 4.7.2]{ceccherini}.
\end{rem}
Whyte used the vanishing criterion Theorem~\ref{thm:deg0} to provide a new proof of a characterisation of amenability for UDBG spaces~\cite[Theorem 7.1]{whyte}. The original result is due to Block and Weinberger~\cite[Theorem 3.1]{blockweinberger} and it is one of the key applications of uniformly finite homology. 
\begin{thm}[characterisation of amenability]\label{thm:charamenability}
Let $X$ be a UDBG space. The following are equivalent:
\begin{enumerate}
\item The UDBG space~$X$ is non-amenable.
\item We have~$\huf_0(X;\R)=0$.
\item We have~$\huf_0(X;\Z)=0$.
\item We have~$\fclz X = 0$.
\end{enumerate}
\end{thm}

Theorem~\ref{thm:whyterigidity} and
Theorem~\ref{thm:charamenability} imply the following rigidity
result for non-amenable spaces:

\begin{cor}\label{cor:rignonamenable}
  Any quasi-isometry between non-amenable UDBG spaces is uniformly close to a bilipschitz equivalence.
\end{cor}
This answered a question originally stated by Gromov~\cite{gromov1}, which was answered by several authors using different tools~\cite{bogopolski, papasoglu, delaharpe}.


\medskip
\vfill

\noindent
\emph{Francesca Diana}\\[.5em]
  {\small
  \begin{tabular}{@{\qquad}l}
    \textsf{francesca.diana@mathematik.uni-regensburg.de}\\
    \textsf{http://homepages-nw.uni-regensburg.de/$\sim$dif13273/}
  \end{tabular}}\\[.5em]
\emph{Clara L\"oh}\\[.5em]
  {\small
  \begin{tabular}{@{\qquad}l}
    \textsf{clara.loeh@mathematik.uni-regensburg.de}\\
    \textsf{http://www.mathematik.uni-regensburg.de/loeh}\\
    Fakult\"at f\"ur Mathematik\\
    Universit\"at Regensburg\\
    93040 Regensburg\\
    Germany
  \end{tabular}}
\end{document}